\documentclass{amsart}
\usepackage{amsmath,amsthm}
\usepackage{pgfbaseimage} 

\newtheorem{thm}{Theorem}[section]

\theoremstyle{definition}

\theoremstyle{remark}

\numberwithin{equation}{section}

\begin{document}

\title{Information Theory and Moduli of Riemann Surfaces}%:
%Algebraic Geometry Meets Compressed Sensing}


%    author one information
\author{James S.~Wolper}
\address{Dept.~of Mathematics, Idaho State University, 921 S.~8th
Ave., Mail Stop 8085, Pocatello, ID~~83209~USA}
\email{wolpjame@isu.edu}
%\thanks{}

%    \subjclass is required.

\date{4 September 2013}

%\dedicatory{}

%    "Communicated by" -- provide editor's name; required.
%\commby{}

%    Abstract is required.
\begin{abstract}
%A random change-of-basis for the space of holomorphic
%differentials on a Riemann Surface of genus $g$ determines a set of
%$3g-3$ moduli, with high probability.  

One interpretation of Torelli's Theorem, which asserts that a 
compact Riemann Surface $X$ of genus $g > 1$ is determined
by the $g(g+1)/2$ entries of the  period matrix, is that the period matrix is
a message about $X$.  Since
this message depends on only $3g-3$ moduli, it is sparse, or at least
approximately so, in the sense of information theory.  
Thus, methods from information
theory  may be useful in reconstructing the period
matrix, and hence the Riemann surface, from a 
small subset of the periods.  The results here show
that, with high probability, any set of $3g-3$ periods form moduli
for the surface.
\end{abstract}

\maketitle
%%%%%%%%%%%%%%%%%%%%%%%%%%%%%
\renewcommand{\sectionmark}[1]{}

%%%%%%%%%%%%%%%%%%%%%%%%%%%%%%%%%%%%%%%%%%%
\section{Introduction}

The connection between Information Theory and Algebraic Geometry,
once unthinkable, was established by the rise of applications
like Goppa codes \cite{TsV} and elliptic-- and hyperelliptic
curve cryptography \cite{CF}.  

Both coding and cryptography
are branches of information theory as originally formulated by
Shannon \cite{Sh}.  A \emph{message} is a sequence of symbols
from a known alphabet with known probabilities
(to be concrete, imagine a stream of
bits selected from $\{{\tt 0}, {\tt 1}\}$ with known probability
of choosing each).  The information content of such a 
message is the smallest number of bits required to send
it; this is  related to its Kolmogorov complexity.
In coding, the goal is to adjoin bits, that is,
redundancy, to make the message more likely to understand;
in cryptography, one of the goals is to hide redundancy.

Goppa codes are constructed from the 
line bundles and points
of an algebraic curve defined over a finite field; the field
elements constitute the alphabet.  The parameters of the code
are determined using the Riemann--Roch theorem.

The most common cryptographic constructions use the Jacobian
of an algebraic curve, most often an elliptic or
hyperelliptic curve; the cryptographic strength
comes from the difficulty of solving the discrete logarithm
problem in these groups.

One generally discusses cryptographic protocols using a sender
Alyce sending a message to a receiver Bob that may be intercepted by
an eavesdropper Eve.  Now suppose that Alyce wants to describe
a compact Riemann surface with genus $g > 1$ to Bob.  By Torelli's
theorem she can, in principle, send Bob a period matrix of the curve
(although in practice it can be difficult to tell much about
the curve from its period matrix).  This message consists 
of $g(g+1)/2$ complex numbers, or their approximations.  However,
the dimension of the moduli space of compact 
Riemann surfaces of genus $g > 1$ is only $3g-3$, so in principle,
or locally,
the information content of Alyce's message is much
smaller than its length.

In the language of information theory, then, the message
is {\sl compressible\/}.  This concept is familiar to computer
users who use utilities like {\tt gzip} to compress files 
and save disk space.

Another perspective on this situation is that the message
describing the Riemann surface is {\sl sparse\/}, or
approximately so.  To illustrate what sparsity means in this context,
imagine some finite energy signal (say, a voice reading
aloud) and its Fourier transform.  In many cases drawn
from nature (see \cite{CT}) the set of coefficients of
this expansion (or other expansion, such as wavelets, 
or even non-orthonormal expansions \cite{D})  contains only
a few large numbers, while the rest of the coefficients are
zero or small; this is the meaning of sparse in this
context.  Put differently, the spectrum of the signal 
is concentrated in a few bands.

The essence of many compression and de--noising schemes
is to look at the spectrum of a signal and approximate
it using only the terms with large coefficients. 
\emph{Compressed Sensing} \cite{D} takes a completely different
approach; rather than collecting the whole signal before
compressing,  compressed sensing collects a smaller number of
samples, in effect compressing at the source.  One way to
do this might be take a random set of Fourier coefficients.
For example, given $x$, the discrete Fourier transform of
some signal, let $A$ be a \emph{random}
 $m\times n$ matrix with entries in $\{0, 1\}$, and let $y = Ax$;
$y$ is the observation.  In theory, the problem of
reconstructing $x$ from $y$ is not well--defined, but
in practice when the signal is sparse excellent approximations
and even exact reconstructions are possible.

It seems natural that at least the ideas of compressed
sensing might apply to Alyce's description of a
Riemann Surface.

Furthermore, when Eve hears Alyce's list of $g(g+1)/2$ complex
numbers, it is natural for her to ask whether they
are the periods of a compact Riemann Surface.  This is
just a restatement of the classical Schottky problem,
asking how to distinguish period matrices from
ordinary elements of the Siegel upper half
space.  Call this version the
{\sl Information--Theoretic Schottky Problem\/}.

Eve might also wonder about the properties of the Riemann
Surface represented by Alyce's message; for example,
she might ask whether the surface in question is
hyperelliptic, or trigonal, or whether it has
a non-trivial automorphism group, {\it etc.\/} Call this the 
{\sl Information--Theoretic Torelli Problem\/}.
In this case, numerical results \cite{W12a} 
(and unpublished analytical results) indicate 
that one can distinguish the period matrices of
hyperelliptic surfaces among arbitrary period matrices
by examining the \emph{statistics} of the periods:
the periods are \emph{band--limited} in the sense
that their frequencies (arguments) are tightly clustered.

%%%%%%%%%%%%%%%%%%%%%%%%%%%%%%%%%%%%%%%%%%%%%%%%
\section{Period Matrices}

Begin by fixing notation; consult
 \cite{GH} as a general reference.
Let $X$ be a compact Riemann Surface 
of genus $g > 1$; equivalently,
$X$ is a nonsingular complex algebraic curve.  Choose a basis $\omega_1,
\ldots, \omega_g$ for the space $H^{1,0}(X)$ of holomorphic differentials
on $X$, and a symplectic basis $\alpha_1, \ldots, \alpha_g, \beta_ 1,
\ldots, \beta_g$ for the singular homology $H_1(X, {\bf Z})$, normalized
so $\int_{\alpha_i} \omega_j = \delta_{ij}$, the Dirac delta.  The 
matrix $\Omega_{ij} := \int_{\beta_i} \omega_j$ is the {\sl period
matrix\/}; Riemann proved that it is symmetric with positive definite imaginary
part.  The torus ${\mathbb C}^g / [I | \Omega]$ is the {\sl Jacobian\/} of $X$.
Torelli's Theorem asserts that the Jacobian determines all of the
properties of $X$. In practice deciding which properties apply
is seldom successful, and any success (like \cite{W07}) depends on
deep tools like Riemann's theta--function.

The period matrix is symmetric with positive-definite imaginary part,
and the space of such matrices forms the {\sl Siegel upper half-space\/} 
${\mathcal H}_g$.
Its dimension is $g(g+1)/2$, while the dimension of the moduli space of 
curves of degree $g$ is $3g-3$.  Distinguishing the period
matrices from arbitrary elements of ${\mathcal H}_g$ is the {\sl Schottky
Problem\/}.  

It would be unusual -- and noteworthy -- were a  problem as 
old as the Schottky problem
to remain unsolved.  Grushevsky's survey \cite{G} contains more detail,
but previous solutions have involved $\theta$--function
identities, the geometry of the $\Theta$ divisor, Kummer varieties,
and solitons; the wide range of techniques indicates
that the Schottky problem is a central problem in mathematics.
The information-theoretic approach deepens this idea.

The aim now
is to study and then exploit the in\-form\-ation-theoretic properties of
period matrices.  One interpretation
of Torelli's
Theorem is that  the period matrix is  a {\sl message\/}
or {\sl signal\/} about
$X$.  This message is highly compressible, while a random element
of ${\mathcal H}_g$ is not.

The result here can also be interpreted as a new result in Compressed
Sensing (described below).  Typically, compressed sensing
works with a {\sl linear\/} relationship between the signal
and the measurement, but in the Schottky problem the relationship
between the signal -- that, is, the period matrix -- and the 
measurement -- that is, explicit moduli -- is nonlinear.
Nonetheless, a compressed sensing manipulation of a
{\sl precursor\/} to the signal (here, the holomorphic
differentials) still leads to a compressed sensing result,
at least with high probability.

Nonlinear compressed sensing is an emerging field; the few
results to appear so far include \cite{IM} and \cite{OYDVS}.

%%%%%%%%%%%%%%%%%%%%%%%%%%%%%%%%%%%%%%%%%%%
\section{Compressive Sensing}

{\sl Compressed\/} (or {\sl Compressive)} {\sl Sensing\/} \cite{D} 
is the study of taking 
small representations or samples of sparse (or nearly sparse) signals 
without any loss of information.  A familiar example of compression is the 
Joint Photographic Experts Group (JPEG) standard for storing
and transmitting images, which can encode an image comprising
millions of pixels into a few kilobytes with minimal loss of quality.
Following \cite{D} and \cite{S}, in compressed sensing the compression
happens at the sensor rather than in post-processing.
``Sparse" has many meanings;  one of the simpler ones is
that the signal can be constructed from some basis ({\it e.g.\/},
Fourier series, impulses, wavelets, $\ldots$) with only a few non--zero
coefficients.  

Many compression methods have the following outline:
transform the signal with respect to some basis functions
$\Psi_i$, remove most of the ``unimportant" terms 
({\it eg\/}, those with the smallest coefficients),
and finally perform the inverse transform.  For 
example one can do this with Fourier
coefficients, and the Shannon-Nyquist Sampling Theorem  \cite{N}
determines which coefficients one needs to retain for perfect
or {\sl lossless\/} compression.  The same general outline applies
to wavelets or other basis functions (which need not be orthonormal).
Another concept of sparsity is ``lying on a submanifold (or
subvariety) of small dimension."  \cite{D} discusses other 
definitions.

Abstractly, this corresponds to a signal vector $\vec{x}$
and a linear set of observations $\vec{y}$ formed using
a {\sl sensing matrix\/} $M$, {\it i.e.\/},
$\vec{y} = M\vec{x}.$
In many applications it suffices to use a {\sl random\/}
matrix $M$ in which each entry is the result of a Bernoulli
process of fixed probability \cite{CT}.

Compression works when the signal under consideration is 
sparse, or approximately so.   This is the case with a period
matrix, which, by Torelli's Theorem, represents a compact
Riemann surface completely: while the matrix  
has $g(g+1)/2$ distinct entries, the Riemann surface
it encodes depends on $3g-3$ moduli.   

In this work, the period matrix is treated as a signal $\Pi$ of length
$g(g+1)/2$ which depends on $3g-3$ unknown underlying
parameters.  This is a list of complex numbers; the matrix
structure is in some sense irrelevant becase their placement in
the matrix depends on choices of bases 
for $H^0(K)$ and $H_1(X, \mathbb{Z})$.  Instead of multiplying
$\Pi$ by a random measurement matrix, 
as would be done in compressed sensing, a random change--of--basis
is applied to $H^0(K)$.  With high probability, the first three
rows of the period matrix constructed from the new basis constitute
a set of moduli for $X$, in a sense explained below.

%%%%%%%%%%%%%%%%%%%%%%%%%%%%
\section{Period Matrices and Moduli}

The primary tool relating period matrices to moduli is the following
theorem of Rauch \cite{R}.  Let $K$ denote the canonical 
divisor on $X$.

\begin{thm}\label{thm-rauch} [Rauch]
Let $\{\zeta_1,\ldots \zeta_g\}$ be a normalized basis for $H^{(1,0)}(X)$
of a non-hyperelliptic Riemann surface $X$,
and suppose that $\{\zeta_i \zeta_j : (i,j) \in (I,J)\}$
form a basis for the quadratic differentials $H^0(X, 2K)$.
If another Riemann surface $X'$ has the same entries as $X$ in the
$(I,J)$ positions of its period matrix then $X$ and $X'$ are 
holomorphically equivalent.
\end{thm}

In other words, some sets of $3g-3$ periods form local 
coordinates on the moduli space of compact Riemann
Surfaces.  However, except in special cases like
smooth plane curves (see \cite{W12b}), there are no results
indicating {\sl which\/} sets of $3g-3$ entries constitute moduli.
(In the case of plane curves, one can use the explicit basis for 
$H^0(K)$ to determine sets of moduli.)

%%%%%%%%%%%%%%%%%%%%%%%%%%%
\section{Numerical Experiments}

Numerical experiments
with period matrices of large genus show that the distribution of
the squared absolute values of the periods is consistent
with a compressible signal. These distributions appear to follow a 
power law (the $n^{\rm th}$ entry is bounded by $1/n^p$ for
some power $p$). 

For example, consider the periods of the Fermat
Curve of degree 11, given in homogeneous
by $X^{11}+Y^{11}+Z^{11} = 0$.  Using Maple, the 
squared moduli of the periods show the characteristic 
shape of a compressible signal.

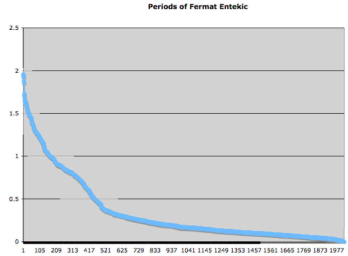
\begin{center}
\pgfdeclareimage[height=3.25in]{Fermat11}{FermatDegree11}
 \pgfuseimage{Fermat11}
\end{center}

Other numerical experiments have exhibited unusual phenomena
in the distribution of the periods of special clases of
curves; in particular, the periods of hyperelliptic 
curves appear to be ``band--limited".  See \cite{W12a}.

%%%%%%%%%%%%%%%%%%%%%%%%%%%
\section{Main Theorem}

This work takes a probabilistic approach to the problem of determining
which sets of $3g-3$ entries form moduli.  The theorem of Rauch above
reduces this to finding a set of $3g-3$ quadratic differentials that form
a basis of $H^0(2K)$.

\begin{thm}\label{thm-main}
Let $X$ be a compact Riemann surface of genus $g$ and
choose a symplectic homology basis for $H^0(K)$.
Let $M$ be a $g\times g$ matrix whose entries lie in $\{0,1\}$ , and are determined
independently by a Bernoulli process.  Use $M$ as a change-of-basis
matrix for $H^0(K)$.   Then after the change--of--basis
the first three rows of the period
matrix for $X$ form moduli for $X$, with high probability.
\end{thm}

\begin{proof}  Using the notation from Theorem \ref{thm-rauch},
let $\{\zeta_1,\ldots \zeta_g\}$ be a normalized basis for $H^{(1,0)}(X)$
and suppose that $\{\zeta_i \zeta_j : (i,j) \in (I,J)\}$ form
a basis for $H^0(2K)$.  The elements of $(I,J)$ are unknown.

Now, let $\omega_i = \sum b_{ij}\zeta_j$, where $b_{ij}$
is a random binary determined by a Bernoulli process.  
Determining the probability that $\Omega = \{\omega_i\}$ forms a basis,
which is the probability that the matrix $[b_{ij}]$ is nonsingular,
is rather subtle.  Tao and Vu \cite{TV}, conjecture the probability to be
\begin{displaymath}
1 - \left ( \frac{1}{2} +o(1) \right )^p,
\end{displaymath}
and prove it greater than
\begin{displaymath}
1 - \left ( \frac{3}{4} +o(1) \right )^p.
\end{displaymath}
In any case, it is very likely that $\Omega$ forms a basis for $H^0(K)$.

Now, assuming that $\Omega$ forms a basis for $H^0(K)$, the question
becomes determining the probability
that 
$\{
\omega_1\omega_1, \ldots, \omega_1\omega_g,
\omega_2\omega_2, \ldots, \omega_2\omega_g,
\omega_3\omega_3, \ldots, \omega_3\omega_g
\}$
forms a basis for $H^0(2K)$.  Notice that this set contains
$3g-3$ elements because of the symmetry of
the period matrix.  

Again, because the coefficients
are binary, and, in particular, positive, it suffices for each
$\zeta_i\zeta_j$ for $(i,j)\in (I,J)$ to have a non-zero
coefficient in one expansion of the products of the $\omega$s.

To find the probability that every coefficient of $\zeta_i\zeta_j$ vanishes
in the expansion of $\omega_\ell\omega_m$, notice that
all four coefficients $b_{i,\ell}$, $b_{j,\ell}$, $b_{i,m}$ and$b_{j,m}$
must vanish, so the probability is $1/16$.  For this to happen for
all coefficients, then, the probability is $(1/16)^{3g-3}.$  And, since
this must happen independently for
all $\zeta_i\zeta_j$, the probability of a basis for $H^0(2K)$ is
\begin{displaymath}
\frac{3g-3}{16^{3g-3}}.
\end{displaymath}

\end{proof}

%%%%%%%%%%%%%%%%%%%%%%%%%%%
\section{Complements}

The proof of the main theorem is deceptively simple, 
because it depends on Rauch's deep theorem.  

Compressed Sensing usually depends on
{\sl linear\/} measurement, and there are few results
about nonlinear measurement processes \cite{S}.
Here, the ``measurement" is linear at the level
of the holomorphic differentials, but the resulting transformation
of the period matrix is non--linear.

A direct relationship between the moduli space of compact
Riemann Surfaces and the period matrix is a long--standing
difficult problem.  One possibility here is to look at a discrete
version of Rauch's Theorem, whose main construction is the
minimal energy element of the homotopy class of 
differentiable maps between compact Riemann surfaces.

In practice, as D.~Litt pointed out in a conversation, 
it may not be possible to transmit periods in a finite message, 
although many complex numbers do have compact descriptions ({\it eg\/},
Gaussian rationals, surds). In other cases it may only be possible to 
transmit an approximation of the periods,
in which case the conclusion is that the curve in 
question is close (in an analytic sense) 
to the Schottky locus, which is already significant.

%%%%%%%%%%%%%%%%%%%%%%%%%%%%%%%%

\end{document}